\documentclass{amsart}

\usepackage{color}
\usepackage{amsmath}
\usepackage{amscd,amsfonts,amssymb}
\usepackage[all]{xy}
\usepackage{mathtools}
\usepackage[colorlinks=true]{hyperref}

\newtheorem{thm}{Theorem}[section] 
\newtheorem{prop}[thm]{Proposition}
\newtheorem{lem}[thm]{Lemma}
\newtheorem{cor}[thm]{Corollary}

\theoremstyle{definition} 
\newtheorem{definition}[thm]{Definition} 
\newtheorem{conj}[thm]{Conjecture}

\theoremstyle{remark} 
\newtheorem{remark}[thm]{Remark} 

\renewcommand{\div}{{\operatorname{div}}}
\newcommand{\Spec}{{\operatorname{Spec}}}
\begin{document}

\title{Arithmetic Demailly Approximation Theorem}
\author{Binggang Qu}
\email{qubinggang22@bicmr.pku.edu.cn}
\address{Beijing International Center for Mathematical Research, Peking University, Beijing, China.}
\author{Hang Yin}
\email{yinhang201@mails.ucas.ac.cn}
\address{Academy of Mathematics and System Science, Chinese Academy of Science, Beijing, China.}
%
\keywords{Arakelov geometry, complex geometry, Demailly approximation, Bertini theorem, Hermitian line bundles, adelic line bundles, height, essential minimum, pseudo-effectivity.}

\begin{abstract}
We generalize the Demailly approximation theorem from complex geometry to Arakelov geometry.

As an application, let $X/\mathbb{Q}$ be an integral projective variety and $\overline N$ be an adelic line bundle on $X$, we prove that $\operatorname{ess}(\overline N) \geq 0$ $\Longrightarrow $ $\overline N$ pseudo-effective. This was proved in \cite{ballay_successive_2021}, assuming $\overline{N}$ relatively semipositive.

We show in the appendix that the above assertion is also true for adelic line bundles on quasi-projective varieties, under the framework of \cite{yuan_adelic_2022}.
\end{abstract}

\maketitle

\vspace*{6pt} \setcounter{tocdepth}{1} \tableofcontents 

\section{Introduction}
\label{sec:introduction}

\subsection{Background}
Let $X/\mathbb{Q}$ be an integral projective variety and $\overline{L}$ be an adelic line bundle. $\overline{L}$ induces an height function $h_{\overline{L}}$ on $X(\overline{\mathbb{Q}})$, which is of central importance in the study of Diophantine geometry.

It is natural to believe that the positivity of the height function $h_{\overline{L}}$ is correlated with the positivity of the line bundle $\overline{L}$. In fact, in \cite{burgos_gil_successive_2015} Burgos Gil, Philippon and Sombra confirmed that $\operatorname{ess}(\overline N) \geq 0$ $\Longleftrightarrow $ $\overline N$ pseudo-effective in toric cases. Later, Ballay showed for general (non-toric) varieties that
\begin{thm}[\cite{ballay_successive_2021}, Theorem 1.1]
	Let $X/\mathbb Q$ be an integral projective variety and $\overline N$ be an adelic line bundle. Assume $\overline N$ is relatively semipositive and $N$ is big, then $\operatorname{ess}(\overline N) \geq 0$ $\Longleftrightarrow $ $\overline N$ pseudo-effective.
\end{thm}

\begin{remark}
	In fact, with the condition ``$N$ is big'' Ballay was able to prove $\overline N$ pseudo-effective $\Longrightarrow$ $\operatorname{ess}(\overline{N}) \geq 0$ and with the condition ``$\overline{N}$ relatively semipositive'' he was able to prove $\operatorname{ess}(\overline{N}) \geq 0$ $\Longrightarrow$ $\overline N$ pseudo-effective.
\end{remark}

In a private conversation, Yuan made the following conjecture that both conditions are unnecessary.
\begin{conj}[Yuan]
	Let $X/\mathbb{Q}$ be an integral projective variety and $\overline N$ be an adelic line bundle. Then $\operatorname{ess}(\overline N) \geq 0$ $\Longleftrightarrow $ $\overline N$ pseudo-effective.
\end{conj}

In this article, we will develop a new arithmetic Bertini type theorem and use it to prove (half of) Yuan's conjecture that $\operatorname{ess}(\overline N) \geq 0$ $\Longrightarrow $ $\overline N$ pseudo-effective.

\subsection{Function field analogue}
Let us first consider the function field analogue, which follows after BDPP theorem and Bertini theorem.
\begin{thm}
	Let $\pi: \mathcal{X} \longrightarrow S$ be an integral projective variety, fibered over a curve $S$, over the field of complex numbers $\mathbb{C}$. Let $\mathcal{N}$ on $\mathcal{X}$ be a line bundle. Then $\operatorname{ess}(\overline{N}) \geq 0 \Longrightarrow$ $\mathcal{N}$ pseudo-effective.
\end{thm}

\begin{proof}
	Say $\dim \mathcal{X}=n+1$. Suppose conversely that $\mathcal{N}$ is not pseudo-effective. Then there exists $\varepsilon>0$ s.t. $\mathcal{N}(\varepsilon)$ is also not pseudo-effective. Here $\mathcal{N}(\varepsilon) = \mathcal{N} + \varepsilon \, \pi^* D$ where $D$ is any divisor on $S$ of degree $1$. Then by BDPP theorem, and because essential minimum is birationally invariant, we may assume there exists an ample line bundle on $\mathcal{X}$ (rather than on some $\mathcal{X}^\prime$ birational to $\mathcal{X}$) such that $\mathcal{L}^n \cdot \mathcal{N}(\varepsilon)<0$.
	
	Let $X \longrightarrow \Spec(\mathbb{C}(S))$ be the generic fiber of $\pi: \mathcal{X} \longrightarrow S$. For any closed subset $Z \subseteq X$, by Bertini theorem, there exists a section $s$ of $\mathcal{L}$ such that $\operatorname{div}(s)$ is irreducible (hence horizontal), $\operatorname{div}(s)_{\mathbb{C}(S)}$ is smooth and $\operatorname{div}(s)_{\mathbb{C}(S)}$ does not meet $Z$. Repeat this procedure $n$ times, we produce a horizontal curve $C$ that does not meet $Z$ and $C \cdot \mathcal{N}(\varepsilon)<0$. Since $Z$ is arbitrary, we see that $\operatorname{ess}(\mathcal{N}(\varepsilon)) = \operatorname{ess}(\mathcal{N}) + \varepsilon \leq 0$. Then $\operatorname{ess}(\mathcal{N}) \leq -\varepsilon$, contradiction.
\end{proof}

\subsection{Motivation}
Now back to number fields, the idea is to mimic the above proof in function field cases.

Suppose that $\mathcal{X}/\mathbb{Z}$ is an arithmetic surface and $\overline{\mathcal{N}}$ is a Hermitianl line bundle. Again, to show $\mathrm{ess}(\overline{\mathcal{N}})\geq 0$ implies $\overline{\mathcal{N}}$ pseudo-effective, by arithmetic BDPP theorem of Ikoma (\cite{ikoma_concavity_2015}, Theorem 6.4), essentially we need to show that $\overline{\mathcal{L}}\cdot \overline{\mathcal{N}}\geq 0$ for any ample Hermitian line bundle $\overline{\mathcal{L}}$.

Now if we have a sequence of good sections $s_n \in H^0(\mathcal{X},n\mathcal{L})$ (which will be produced by our arithmetic Bertini theorem), such that
\begin{itemize}
	\item $\tfrac1n \log\|s_n\| \longrightarrow 0$ in $L^1$-sense.
	\item $\div(s_n)$ is irreducible (then $\div(s_n)$ is horizontal, say $\div(s_n)=\big\{x_n \big\}$ where $x_n$ is a closed point of the generic fiber $\mathcal{X}_\mathbb{Q}$).
\end{itemize}

Then we use these $s_n$ to compute the arithmetic intersection number and see
\begin{flalign*}
	\quad \quad \quad 
	n\overline{\mathcal{L}} \cdot \overline{\mathcal{N}} &= \widehat{\deg} \left ( \overline{\mathcal{N}}|_{\operatorname{div}(s_n)} \right ) - \int_{\mathcal{X}(\mathbb{C})} \log\|s_n\| \; c_1(\overline{\mathcal{N}}).
	&&
\end{flalign*}
Note that $x_n$ is of degree $n\deg(\mathcal{L}_\mathbb{Q})$ and $h_{\overline{\mathcal{N}}}(x_n) = \widehat{\deg} \left ( \overline{\mathcal{N}}|_{\operatorname{div}(s_n)} \right ) / \deg(x_n)$, so divide by $n\deg(\mathcal{L}_\mathbb{Q})$ on both sides we see
\begin{flalign*}
	\quad \quad \quad &
	\frac{\overline{\mathcal{L}} \cdot \overline{\mathcal{N}}}{\deg(\mathcal{L}_\mathbb{Q})} = h_{\overline{\mathcal{N}}}(x_n) - \frac1{n\deg(\mathcal{L}_\mathbb{Q})} \int_{\mathcal{X}(\mathbb{C})} \log\|s_n\| \; c_1(\overline{\mathcal{N}}).
	&&
\end{flalign*} By assumption $\tfrac1n \log\|s_n\| \longrightarrow 0$ in $L^1$-sense so when $n \longrightarrow \infty$, the second term in the right hand side tends to zero and we have
\begin{flalign*}
	\quad \quad \quad &
	\frac{\overline{\mathcal{L}} \cdot \overline{\mathcal{N}}}{\deg(\mathcal{L}_\mathbb{Q})} = \liminf_{n \longrightarrow \infty} h_{\overline{\mathcal{N}}}(x_n) \geq \operatorname{ess}(\overline{\mathcal{N}}) \geq 0.
	&&
\end{flalign*}
The latter inequality is because $\mathcal{X}_\mathbb{Q}$ is a curve so the sequence $\big\{ x_n \big\}$ is automatically generic (meaning that it does not have a subsequence contained in a proper closed subvariety). We win.

\subsection{Main theorems}
\begin{thm}[arithmetic Demailly approximation, Theorem \ref{full arithmetic Demailly approximation}]
	Let $\mathcal{X}/\mathbb{Z}$ be an arithmetic variety, normal and generically smooth. Let $\overline{\mathcal{L}}$ be an ample Hermitian line bundle. Let $Y$ be a closed subset of $\mathcal{X}_\mathbb{Q}$.
	Then there exists $s_{n_i} \in H^0(\mathcal{X},n_i\mathcal{L})$ such that
	\begin{itemize}
		\item $\|s_{n_i}\|_\infty \longrightarrow 1$ and $\frac1{n_i} \log\|s_{n_i}\| \longrightarrow 0$ in $L^1$-topology.
		\item $\operatorname{div}(s_{n_i})_\mathbb{Q}$ is smooth and does not contain any irreducible components of $Y$.
		\item $\operatorname{div}(s_{n_i})$ has no vertical components.
	\end{itemize}
\end{thm}

As an application, we prove
\begin{thm}[Theorem \ref{main}]
	Let $\mathcal{X}/\mathbb{Z}$ be an arithmetic variety of relative dimension $d$ and $\mathcal{\overline{N}}$ be a Hermitian line bundle. Then
	\begin{flalign*}
		\quad \quad \quad &
		\operatorname{ess}(\mathcal{\overline{N}}) \leq \inf \frac{\mathcal{ \overline{L}}^d \cdot \pi^*\mathcal{\overline{N}}}{\mathcal{L}_\mathbb{Q}^d},
		&&
	\end{flalign*}
	where the infimum is taken over all 
	\begin{itemize}
		\item $\pi: \mathcal{X}^\prime \longrightarrow \mathcal{X}$ is a birational morphism with $\mathcal{X}^\prime$ normal and generically smooth.
		\item $\mathcal{\overline{L}}$ is an ample Hermitian line bundle on $\mathcal{X}^\prime$.
	\end{itemize}
	\label{intro_main}
\end{thm}

Theorem \ref{intro_main} plus the arithmetic BDPP theorem of Ikoma (\cite{ikoma_concavity_2015}, Theorem 6.4) gives
\begin{thm}[Corollary \ref{main2}]
Let $X/\mathbb{Q}$ be an integral projective variety and $\overline N$ be an adelic line bundle. Then $\operatorname{ess}(\overline N) \geq 0$ $\Longrightarrow $ $\overline N$ pseudo-effective.
\end{thm}

In appendix we prove that this is also true for adelic line bundles on quasi-projective varieties.
\begin{thm}[Theorem \ref{mainthm_quasi-projective}]
	Let $X/\mathbb{Q}$ be an integral quasi-projective variety and $\overline N$ be an adelic line bundle. Then $\operatorname{ess}(\overline N) \geq 0$ $\Longrightarrow $ $\overline N$ pseudo-effective.
\end{thm}

\subsection{Comparison with other arithmetic Bertini type theorems}
In \cite{moriwaki_arithmetic_1995}, Moriwaki proved
\begin{thm}[Moriwaki]
	Let $\mathcal{X}/\mathbb{Z}$ be an arithmetic variety that is generically smooth, and let $\overline{\mathcal{L}}$ be an ample Hermitian line bundle. Then for (not necessarily closed) points $x_1,\dots,x_q$ on $\mathcal{X}$, there exists a large $m$ and a section $s \in H^0(\mathcal{X},m\mathcal{L})$ s.t.
	\begin{enumerate}
		\item $\|s\|_\infty <1$.
		\item $\div(s)_\mathbb{Q}$ is smooth over $\mathbb{Q}$;
		\item $s(x_i) \neq 0$ for each $i$;
	\end{enumerate}
\end{thm}

Moriwaki's theorem was later generalized by Ikoma \cite{ikoma_bertini-type_2015} to general arithmetic linear series (rather than the full $\bigoplus_n H^0(\mathcal{X},n\mathcal{L})$).

A vast generalization is made in \cite{charles_arithmetic_2021}, where Charles proved
\begin{thm}[Charles]
	Let $\mathcal{X}/\mathbb{Z}$ be an arithmetic variety that is generically smooth, and let $\overline{\mathcal{L}}$ be an ample Hermitian line bundle. Let $Y \subseteq \mathcal{X}_\mathbb{Q}$ be a closed subset. Then the set
	\begin{flalign*}
		\quad \quad \quad 
		E_n = \big\{ &s \in H^0_\mathrm{Ar}(\mathcal{X},n\overline{\mathcal{L}}): \div(s) \text{ is irreducible}, \\
		 &\div(s)_\mathbb{Q} \text{ is smooth and does not contain any irreducible component of $Y$} \big\}
		&&
	\end{flalign*}
	has density tends to $1$ (i.e. $\#E_n / \#H^0_\mathrm{Ar}(\mathcal{X},n\overline{\mathcal{L}}) \longrightarrow 1$). Here under Charles' notation,
	\begin{flalign*}
		\quad \quad \quad &
		H^0_\mathrm{Ar}(\mathcal{X},\overline{\mathcal{L}}) = \big\{ s \in H^0(\mathcal{X},\mathcal{L}): \|s\|_\infty \leq 1 \big\}.
		&&
	\end{flalign*}
\end{thm}

One sees that all existing results are imposing at infinite places the condition ``$\|s\|_\infty \leq 1$''. This gives control of $\displaystyle \int_{\mathcal{X}(\mathbb{C})} \log\|s\| \; c_1(\overline{\mathcal{N}})$ only when the other line bundle is semipositive at infinity (i.e. $c_1(\overline{\mathcal{N}}) \geq 0$).

The novelty of our result is that, we at infinite places use a new condition ``$\tfrac1n \log\|s_n\| \longrightarrow 0$ in $L^1$-sense''. This controls arithmetic intersection number by making the $\displaystyle \int_{\mathcal{X}(\mathbb{C})} \log\|s\| \; c_1(\overline{\mathcal{N}})$ term vanish and works for arbitrary $\mathcal{\overline{N}}$. This is the reason why we could drop the relatively semipositive condition in Ballay's theorem (\cite{ballay_successive_2021}, Theorem 1.1).

\subsection{Acknowledgements}
We express great gratitude to Ye Tian for the introduction to Arakelov geometry. We are indebted to Xinyi Yuan who communicated his conjecture and ideas to us.

We thank Fusheng Deng, Zhuo Liu and Yaxiong Liu for instructions on the theory of psh functions and Demailly approximation. We thank Xiaozong Wang, Qingyang Han and Wei Xue for inspiring discussions. We thank Fran{\c c}ois Charles and Vincent Geudj for email correspondences.

\section{Conventions and terminology}
\begin{definition}
	An \emph{arithmetic variety} is a projective flat morphism $\pi: \mathcal{X} \longrightarrow \mathbb{Z}$ with $\mathcal{X}$ integral and a \emph{Hermitian line bundle} on $\mathcal{X}$ is the datum $\overline{\mathcal{L}} = (\mathcal{L},\|\cdot\|)$ where
	\begin{itemize}
		\item $\mathcal{L}$ is a line bundle on $\mathcal{X}$;
		\item $\|\cdot\|$ is a smooth Hermitian metric on the line bundle $\mathcal{L}(\mathbb{C})$ on $\mathcal{X}(\mathbb{C})$ that is smooth and invariant under complex conjugation.
	\end{itemize}
\end{definition}

\begin{definition}
	A Hermitian line bundle induces an Arakelov height function
	\begin{flalign*}
		\quad \quad \quad 
		h_{\overline{\mathcal{L}}}: X(\overline{\mathbb{Q}}) &\longrightarrow \mathbb{R} \\
		x &\longmapsto \frac{\widehat{\deg}\left( \overline{\mathcal{L}}|_{\overline{\{x\}}} \right)}{\deg(x)}.
		&&
	\end{flalign*}
\end{definition}

\begin{definition}
	A Hermitian line bundle $\overline{\mathcal{L}}$ on $\mathcal{X}/\mathbb{Z}$ is called \emph{ample} if
		\begin{itemize}
			\item $\mathcal{L}_\mathbb{Q}$ is ample on $\mathcal{X}_\mathbb{Q}$.
			\item $\overline{\mathcal{L}}$ is vertically positive, i.e.
				\begin{itemize}
					\item $c_1(\overline{\mathcal{L}})$ is a positive $(1,1)$-form,
					\item for any vertical curve $C$, $\deg {\mathcal{L}}|_C > 0$.
				\end{itemize}
			\item $\overline{\mathcal{L}}$ is horizontally positive, i.e. for any horizontal closed subvariety $\mathcal{Y}$, $\overline{\mathcal{L}}^{\dim \mathcal{Y}} \cdot \mathcal{Y} > 0$.
		\end{itemize}
\end{definition}

\begin{remark}
	This is equivalent to the definition used by Ikoma \cite{ikoma_concavity_2015} and is stronger than Zhang's original definition, which required only vertically semipositive. We want to be consistent with Ikoma, so that no confusion is made when we invoke his arithmetic BDPP theorem.
\end{remark}

\begin{definition}
Let $X/\mathbb{Q}$ be an integral projective variety. An \emph{adelic line bundle} is the datum $\overline{L} = \big( L, \big\{ \|\cdot\|_p \big\}_{p \leq \infty} \big)$ where
	\begin{itemize}
		\item $L$ is a line bundle on $X$;
		\item $\|\cdot\|_p$ is a continuous metric on $L_p^\text{an}$ on $X_p^\text{an}$,
	\end{itemize}
	and satisfying the \emph{coherence} condition that there exists an open set $U \subseteq \Spec(\mathbb{Z})$ and a model $(\mathcal{X}_U,\mathcal{L}_U)$ of $(X,L)$ over $U$ that induces $\|\cdot\|_p$ for all $p \in U$.
\end{definition}


\begin{definition}
	An adelic line bundle induces an adelic height function
\begin{flalign*}
	\quad \quad \quad 
	h_{\overline{L}}: X(\overline{\mathbb{Q}}) &\longrightarrow \mathbb{R} \\
	x &\longmapsto -\frac1{\deg(x)} \sum_{p \leq \infty} \sum_{x^\prime \in O(x)} \log\|s(x^\prime)\|_p
	&&
\end{flalign*}
where $s$ is any rational section of $L$ non-vanishing at $x$.
\end{definition}

\begin{thm}
	An adelic line bundle is a uniform limit of Hermitian line bundles. To be precise, let $X/\mathbb{Q}$ be an integral projective variety and $\overline{L}= \big(L,\big\{ \|\cdot\|_p \big\}_{p\leq \infty} \big)$ be an adelic line bundle, then there exists an open subset $U \subseteq \Spec(\mathbb{Z})$ and a sequence of Hermitian models $(\mathcal{X}_n,\overline{\mathcal{L}}_n)$ of $(X,L)$ (here $\overline{\mathcal{L}}_n$ are allowed to be $\mathbb{Q}$-Hermitian line bundles) such that
	\begin{itemize}
		\item for $p \in U$, $\|\cdot\|_{n,p} = \|\cdot\|_p$.
		\item for $p \notin U$ and finite, $\|\cdot\|_{n,p} / \|\cdot\|_p \longrightarrow 1$ uniformly. Here $\|\cdot\|_{n,p}$ is the metric induced by the model $(\mathcal{X}_n,{\mathcal{L}}_n)$ at $p$.
		\item for $p=\infty$, $\|\cdot\|_{n,\infty} / \|\cdot\|_{\infty} \longrightarrow 1$ uniformly.
	\end{itemize}
\end{thm}

\begin{definition}
	Let $\overline{L}$ be an adelic line bundle, the set of \emph{small sections} is defined to be
	\begin{flalign*}
		\quad \quad \quad &
		H^0(X,\overline{L}) = \big\{ s \in H^0(X,L): \sup \|s\|_p \leq 1 \text{ for all $p \leq \infty$} \big\}
		&&
	\end{flalign*} and we write $h^0(X,L) = \log\#H^0(X,\overline{L})$.
	Say $\dim X=d$, the \emph{arithmetic volume} is defined to be
	\begin{flalign*}
		\quad \quad \quad &
		\operatorname{vol}(\overline{L}) = \lim_{n \rightarrow \infty} \frac{h^0(X,n\overline{L})}{n^{d+1}/(d+1)!}.
		&&
	\end{flalign*}
\end{definition}

\begin{definition}
	An adelic line bundle $\overline{L}$ is called \emph{big} if $\operatorname{vol}(\overline{L})>0$.
\end{definition}

\begin{definition}
	An adelic line bundle $\overline{L}$ is called \emph{pseudo-effective} if for any big adelic line bundle $\overline{B}$ and any $n>0$, $\overline{L} + \frac1n \overline{B}$ is still big.
\end{definition}

\begin{thm}[arithmetic BDPP theorem, \cite{ikoma_concavity_2015} Theorem 6.4]
	Let $\mathcal{X}/\mathbb{Z}$ be an arithmetic variety and $\overline{\mathcal{N}}$ be a Hermitian line bundle on $\mathcal{X}$. Then $\overline{\mathcal{N}}$ is pseudo-effective if and only if for any birational morphism $\pi: \mathcal{X}^\prime \longrightarrow \mathcal{X}$ such that $\mathcal{X}^\prime$ normal and generically smooth, and for any ample Hermitian line bundle $\overline{\mathcal{L}}$ on $X^\prime$, we have $\overline{\mathcal{L}}^d \cdot \pi^*\overline{\mathcal{N}} \geq 0$.
\end{thm}

\section{Arithmetic Demailly approximation I: infinite places}
\subsection{Demailly approximation theorem in complex geometry}
The following is a baby case of the Demailly approximation theorem (see \cite{demailly_analytic_2010}, Corollary 13.23 or \cite{guedj_degenerate_2017}, Theorem 8.9), but with extra sup-norm estimation. It is not written directly in the literature, so for the readers' convenience we shall give a short proof here. There is nothing original.

\begin{thm}
	Let $X$ be a projective complex manifold and $\overline{L}=(L,\|\cdot\|)$ be an ample line bundle with smooth positive Hermitian metric $\|\cdot\|$. Then there exists $s_{n_i} \in H^0(X,n_iL)$ s.t. $\|s_{n_i}\|_\infty=1$ and $\frac1{n_i} \log\|s_{n_i}\| \longrightarrow 0$ in $L^1$-sense.
	\label{complex Demailly}
\end{thm}

%
%
%
%
%

\begin{proof}
	Let $\sigma_i$, $i=1,\dots,r_n$ be orthonormal basis of $H^0(X,nL)$. Consider the Bergman kernel $b_n(x)=\sum_{i=1}^{r_n} \|\sigma_i(x)\|^2$. By Bouche-Tian,
	\begin{flalign*}
		\quad \quad \quad &
		b_n(x) \sim r_n=\dim H^0(X,nL)
		&&
	\end{flalign*} uniformly on $X$ as $n \longrightarrow \infty$. In particular, $\frac1n \log b_n(x) \longrightarrow 0$. Let $u_n(x)= \max_{i=1}^{r_n} \|\sigma_i(x)\|^2$. Then $1 \leq b_n/u_n \leq r_n$ and hence $\frac1n \log u_n \longrightarrow 0$.
	
	Fix $n$, consider $s_{n,\ell} = \sum_{i=1}^{r_n} \sigma_i^\ell \in H^0(X,n\ell L)$. We have $\frac1{\ell} \log \|s_{n,\ell}\| \longrightarrow \log u_n $ as $\ell \longrightarrow \infty$.
	 
	By Gromov's inequality, there exists $C$ s.t. for any $n$ and any $s \in H^0(X,nL)$ we have
	\begin{flalign*}
		\quad \quad \quad &
		\|s\|_{L^2} < \|s\|_\infty \leq C \sqrt{r_n} \cdot \|s\|_{L^2}.
		&&
	\end{flalign*} So $\|s_{n,\ell}\|_\infty \leq \sum_{i=1}^{r_n} \|\sigma_i\|_\infty^\ell \leq r_n \max_i \|\sigma_i\|_\infty^\ell$ and thus $\|s_{n,\ell}\|_\infty^{\frac1\ell} \leq r_n^{\frac 1\ell} \max_{i=1}^{r_n} \|\sigma_i\|_\infty \leq C r_n^{\frac12 + \frac1\ell}$.
	 
	Take $n_i$ be s.t. $\Big\| \frac1{n_i} \log u_{n_i} \Big\|_{L^1} \leq \frac1{2i}$. For each $n_i$, take $\ell_i$ be s.t. $\Big \|\frac1{n_i\ell_i} \log\|s_{n_i,\ell_i}\| - \frac1{n_i} \log u_{n_i} \Big  \|_{L^1} \leq \frac1{2i}$. Then $s_{n_i,\ell_i} \in H^0(X,n_i\ell_i L)$ satisfies $\frac1{n_i\ell_i} \log\|s_{n_i,\ell_i}\| \longrightarrow 0$ and $\lim_i \|s_{n_i\ell_i}\|_\infty^{\frac{1}{n_i\ell_i}} \leq 1$.
	 
	Note that if $\|s_n\|_\infty^{\frac1n} \longrightarrow c < 1$ then $\frac1n \log \|s_n\| \leq \log c < 0$ can not tends to $0$. This forces $\lim_i \|s_{n_i\ell_i}\|_\infty^{\frac{1}{n_i\ell_i}} \geq 1$ and thus $\lim_i \|s_{n_i\ell_i}\|_\infty^{\frac{1}{n_i\ell_i}} = 1$.
	 
	Then consider $s^\prime_{n_i,\ell_i} = \frac {s_{n_i,\ell_i} } {\|s_{n_i,\ell_i}\|_\infty} \in H^0(X,n_i\ell_i L)$. We have $\|s^\prime_{n_i,\ell_i}\|_\infty=1$ and
	\begin{flalign*}
		\quad \quad \quad &
		\frac1{n_i\ell_i} \log\|s^\prime_{n_i,\ell_i}\| = \frac1{n_i\ell_i} \log\|s_{n_i,\ell_i}\| - \frac1{n_i\ell_i} \log\|s_{n_i,\ell_i}\|_\infty \longrightarrow 0.
		&&
	\end{flalign*}
\end{proof}

\subsection{Real sections}
Let $\mathcal{X}/\mathbb{Z}$ be an arithmetic variety, normal and generically smooth. Let $\overline{\mathcal{L}}$ be an ample Hermitian line bundle.

Since $\mathcal{X}$ is normal, the Stein factorization of $\mathcal{X} \longrightarrow \operatorname{Spec} \mathbb{Z}$ is of the form
\begin{flalign*}
	\quad \quad \quad &
	\mathcal{X} \longrightarrow \operatorname{Spec} O_K \longrightarrow \operatorname{Spec} \mathbb{Z}
	&&
\end{flalign*} where $K$ is a number field and $O_K$ is the ring of integers of $K$.

We know
\begin{flalign*}
	\quad \quad \quad &
	\mathcal{X}(\mathbb{C}) = \coprod_{\sigma: K \rightarrow \mathbb{C}} \mathcal{X}_\sigma(\mathbb{C}),
	&&
\end{flalign*} each $\mathcal{X}_\sigma(\mathbb{C})$ is a connected projective complex manifold and the complex conjugation $F$ acts on $\mathcal{X}(\mathbb{C})$ via the following rule:
\begin{itemize}
	\item if $\sigma$ is a real embedding, then $F$ acts on $X_\sigma(\mathbb{C})$ via $Fx_\sigma=\overline{x}_\sigma$.
	\item if $(\sigma,\overline{\sigma} )$ is a pair of complex embeddings, then $F$ acts on $X_\sigma(\mathbb{C}) \coprod X_{\overline{\sigma}}(\mathbb{C})$ via $F(x,y) = (\overline{y},\overline{x})$.
\end{itemize}

We know
\begin{flalign*}
	\quad \quad \quad &
	H^0(\mathcal{X},n\mathcal{L})_\mathbb{C} = H^0(\mathcal{X}(\mathbb{C}),n\mathcal{L}(\mathbb{C})) = \bigoplus_{\sigma: K \rightarrow \mathbb{C}} H^0(\mathcal{X}_\sigma(\mathbb{C}),n\mathcal{L}_\sigma(\mathbb{C}))
	&&
\end{flalign*} and the complex conjugation $F$ acts on $H^0(\mathcal{X}(\mathbb{C}),n\mathcal{L}(\mathbb{C}))$ via the following rule:
\begin{itemize}
	\item if $\sigma$ is a real embedding, then $F$ acts on $H^0(\mathcal{X}_\sigma(\mathbb{C}),n\mathcal{L}_\sigma(\mathbb{C}))$ via $Fs_\sigma=\overline{s}_\sigma$.
	\item if $(\sigma,\overline{\sigma} )$ is a pair of complex embeddings, \\ then $F$ acts on $H^0(\mathcal{X}_\sigma(\mathbb{C}),n\mathcal{L}_\sigma(\mathbb{C})) \oplus H^0(\mathcal{X}_{\overline{\sigma}}(\mathbb{C}),n\mathcal{L}_{\overline{\sigma}}(\mathbb{C}))$ via $F(s_1,s_2) = (\overline{s}_2,\overline{s}_1)$.
\end{itemize}

Real sections $H^0(\mathcal{X},n\mathcal{L})_\mathbb{R}$ are elements of $H^0(\mathcal{X},n\mathcal{L})_\mathbb{C}$ fixed by complex conjugation. In \\ $\bigoplus_{\sigma: K \rightarrow \mathbb{C}} H^0(\mathcal{X}_\sigma(\mathbb{C}),n\mathcal{L}_\sigma(\mathbb{C}))$, they are elements of the following form:
\begin{itemize}
	\item if $\sigma$ is a real embedding, $\overline{s}_\sigma = s_\sigma$ is a real section.
	\item if $(\sigma,\overline{\sigma} )$ is a pair of complex embeddings, sections over $\sigma$ and $\overline{\sigma}$ are $(s_\sigma, \overline{s}_{\overline{\sigma}})$, i.e. a conjugate pair.
\end{itemize}

$H^0(\mathcal{X},n\mathcal{L})_\mathbb{R}$ inherits sup-norm from $H^0(\mathcal{X},n\mathcal{L})_\mathbb{C}$ and makes itself a normed $\mathbb{R}$-vector space. $H^0(\mathcal{X},n\mathcal{L}) \subseteq H^0(\mathcal{X},n\mathcal{L})_\mathbb{R}$ is a lattice.

\begin{remark}
	We invite readers to take $\mathcal{X}=\Spec O_K$. In this case, a detailed explaination can be found in \cite{neukirch_algebraic_2011}, Chapter I, \textsection 5.

\end{remark}

The goal of this section is to prove Demailly approximation by real sections.

\begin{prop}
	There exist $s_{n_k} \in H^0(\mathcal{X},n_k\mathcal{L})_\mathbb{R}$ such that $\|s_{n_k}\|_\infty=1$ and $\frac1{n_k} \log\|s_{n_k}\| \longrightarrow \nolinebreak 0$ in $L^1$-topology.
	\label{real sections}
\end{prop}

\begin{proof}
	Name the real embeddings as $\sigma_1,\dots,\sigma_r$ and complex embeddings as $\tau_1,\overline{\tau}_1,\dots,\tau_s,\overline{\tau}_s$. Apply Theorem \ref{complex Demailly} to $(\mathcal{X}_{\sigma_i}(\mathbb{C}), \mathcal{L}_{\sigma_i}(\mathbb{C}) )$, $i=1,\dots,r$ and $(\mathcal{X}_{\tau_j}(\mathbb{C}), \mathcal{L}_{\tau_j}(\mathbb{C}) )$, $j=1,\dots,s$. We can find an index sequence $n_k$ and sections at each place
	\begin{itemize}
		\item $s_{n_k,\sigma_i} \in H^0(\mathcal{X}_{\sigma_i}(\mathbb{C}),n_k \mathcal{L}_{\sigma_i}(\mathbb{C}) )$ for $i=1,\dots,r$ such that $\|s_{n_k,\sigma_i}\|_\infty=1$ and $\frac1{n_k} \log\|s_{n_k,\sigma_i}\| \longrightarrow 0$ in $L^1$-topology. (Then $s^\prime_{2n_k,\sigma_i} = s_{n_k,\sigma_i} \cdot \overline{s}_{n_k,\sigma_i} \in H^0(\mathcal{X}_{\sigma_i}(\mathbb{C}),2n_k \mathcal{L}_{\sigma_i}(\mathbb{C}) )$ is real and such that $\|s^\prime_{2n_k,\sigma_i}\|_\infty=1$ and $\frac1{2n_k} \log\|s^\prime_{2n_k,\sigma_i}\| \longrightarrow 0$ in $L^1$-topology.)
		\item $s_{2n_k,\tau_j} \in H^0(\mathcal{X}_{\tau_j}(\mathbb{C}),2n_k \mathcal{L}_{\tau_j}(\mathbb{C}) )$ for $j=1,\dots,s$ such that $\|s_{2n_k,\tau_j} \|_\infty=1$ and $\frac1{2n_k} \log \|s_{2n_k, \tau_j}\| \longrightarrow 0$ in $L^1$-topology.
	\end{itemize}
	Then $s_{2n_k} = (s^\prime_{2n_k,\sigma_1},\dots,s^\prime_{2n_k,\sigma_r}, s_{2n_k,\tau_1}, \overline{s}_{2n_k,\overline{\tau}_1},\dots,s_{2n_k,\tau_s}, \overline{s}_{2n_k,\overline{\tau}_s}) \in H^0(\mathcal{X},2n_k\mathcal{L})_\mathbb{R}$ will \nolinebreak do.
\end{proof}

\begin{remark}
	In the above proof, we performed Demailly approximation simultaneously on several complex manifolds. One can go back to the proof of Theorem \ref{complex Demailly} and see that this can be arranged.
\end{remark}

\subsection{Plurisubharmonic functions}
We briefly recall basic properties of plurisubharmonic (psh) functions that will be used in this article. The reference is \cite{guedj_degenerate_2017}, Chapter VIII.

\begin{definition}
	Let $(X,\omega)$ be a compact K\"{a}hler manifold. A function $f: X \longrightarrow \mathbb{R} \cup \{-\infty\}$ is called \emph{$\omega$-psh} if
	\begin{itemize}
		\item locally in complex topology, $f$ is sum of a psh function and a smooth function.
		\item $dd^cf + \omega$ is a positive $(1,1)$ current.
	\end{itemize}
\end{definition}

The set of all $\omega$-psh functions on $X$ is denoted by $\operatorname{PSH}(X,\omega)$ and is endowed with the $L^1$-topology.
\begin{thm}[\cite{guedj_degenerate_2017} Proposition 8.5]\label{compactness}
	$\operatorname{PSH}_0(X,\omega) = \big\{ f \in \operatorname{PSH}(X,\omega): \sup_X f =0 \big\}$ is compact.
\end{thm}



\begin{lem}
	Let $(X,\omega)$ be a compact K\"{a}hler manifold and $(f_n)$ be a sequence of $\omega$-psh functions. Then $f_n \longrightarrow 0$ in $L^1$-topology if and only if $f_n \longrightarrow 0$ in measure.
	\label{equivalent convergence lemma}
\end{lem}

\begin{remark}
Recall on a probability space $(X,\mu)$, a sequence of measurable functions $(f_n)$ is said to converge to $f$ in measure if for any $\varepsilon>0$,
\begin{flalign*}
 \quad \quad \quad &
 \lim_{n\longrightarrow \infty}\mu\{|f_n-f|>\varepsilon\}=0.
 &&
\end{flalign*}
It is known that if $f_n\longrightarrow f$ in $L^1$, then $f_n\longrightarrow f$ in measure. The coverse is in general false.
\end{remark}

\begin{proof}
	From discussion above, we only prove sufficiency. For any subsequence $(f_{n_i})$ of $(f_n)$, by compactness (Theorem \ref{compactness}), there exists a subsubsequence $(f_{n_{i_j}})$ and constant numbers $(c_{n_{i_j}})$ such that $f_{n_{i_j}} + c_{n_{i_j}} \longrightarrow f$ in $L^1$-topology. In particular, $f_{n_{i_j}} + c_{n_{i_j}} \longrightarrow f$ in measure.

	By assumption $f_n \longrightarrow 0$ in measure, so $c_{n_{i_j}} \longrightarrow f$ in measure. Since $c_{n_{i——j}}$ are constants, $f=c$ is constant almost everywhere. Therefore, $c_n \longrightarrow c$ and $f_{n_{i_j}} \longrightarrow 0$ in $L^1$-topology.

	We have proved that for any subsequence $f_{n_i}$, there exists a subsubsequence $f_{n_{i_j}} \longrightarrow 0$ in $L^1$-topology. This forces the whole sequence $f_n \longrightarrow 0$ in $L^1$-topology.
\end{proof}

\subsection{Integral sections}
\begin{lem}
	Let $X$ be a compact K\"{a}hler manifold and $\overline{L}=(L,\|\cdot\|)$ be a holomorphic line bundle on $X$ with smooth and positive Hermitian metric $\|\cdot\|$. Let $s_n, s_n^\prime \in H^0(X,nL)$ be holomorphic sections such that $\|s_n\|_\infty = 1$ and $\frac1n \log\|s_n\| \longrightarrow 0$ in $L^1$-topology.
	
	If there exists $0<k<1$ such that $\|s_n-s_n^\prime\|_\infty \leq k^n$, then $\|s_n^\prime\|_\infty \longrightarrow 1$ and $\frac1n \log\|s_n^\prime\| \longrightarrow 0$ in $L^1$-topology.
		\label{k^n then ok}
\end{lem}

\begin{proof}
	By Lemma \ref{equivalent convergence lemma} we need only to show $\frac1n \log\|s_n^\prime\| \longrightarrow 0$ in measure, which is obvious.
\end{proof}

Let $\lambda_\mathrm{last}(n\overline{\mathcal{L}})$ be the last Minkowski minimum of the lattice $H^0(\mathcal{X},n\mathcal{L})\subseteq H^0(\mathcal{X},n\mathcal{L})_\mathbb{R}$. One version of Zhang's Nakai-Moishezon is the following
\begin{thm}[Zhang \cite{zhang_positive_1992} \cite{zhang_positive_1995}]
	Let $\mathcal{X}/\mathbb{Z}$ be an arithmetic variety, normal and generically smooth and $\overline{\mathcal{L}}$ be an ample Hermitian line bundle. Then there exists $0<k<1$ such that when $n$ is large, $\lambda_\mathrm{last}(n\overline{\mathcal{L}}) \leq k^n$.
	\label{Zhang Nakai-Moishezon}
\end{thm}

\begin{cor}\label{2}
	There exists $0<k<1$ such that for any real section $s_n \in H^0(\mathcal{X},n\mathcal{L})_\mathbb{R}$, there exists an integral section $s_n^\prime \in H^0(\mathcal{X},n\mathcal{L})$ such that $\|s_n-s_n^\prime\|_\infty \leq k^n$.
	\label{k^n}
\end{cor}

\begin{thm}
	Let $\mathcal{X}/\mathbb{Z}$ be an arithmetic variety, normal and generically smooth. Let $\overline{\mathcal{L}}$ be an ample Hermitian line bundle.
	There exists $s_{n_i} \in H^0(\mathcal{X},n_i\mathcal{L})$ such that $\|s_{n_i}\|_\infty \longrightarrow 1$ and $\frac1{n_i} \log\|s_{n_i}\| \longrightarrow 0$ in $L^1$-topology.
	\label{arithmetic Demailly infinite}
\end{thm}

\begin{proof}
	This is a combination of Proposition \ref{real sections}, Corollary \ref{k^n} and Lemma \ref{k^n then ok}.
\end{proof}

\section{Arithmetic Demailly approximation II: finite places}
Throughout this esction, $k$ is the number in Corollary \ref{2}.

We fix an arithmetic variety $\mathcal{X}/\mathbb{Z}$ and an ample Hermitian line bundle $\overline{\mathcal{L}}$. Then $H^0(\mathcal{X}, n\mathcal{L})_\mathbb{R}$ is a normed $\mathbb{R}$-vector space and $H^0(\mathcal{X}, n\mathcal{L}) \subseteq H^0(\mathcal{X}, n\mathcal{L})_\mathbb{R}$ is a lattice. Let $B_n(s,R)$ denote the ball in $H^0(\mathcal{X}, n\mathcal{L})_\mathbb{R}$ with center $s$ and radius $R$.

In Theorem \ref{arithmetic Demailly infinite} we constructed a sequence of integral sections $s_{n_i} \in H^0(\mathcal{X},n_i\mathcal{L})$ such that $\|s_{n_i}\|_\infty \longrightarrow 1$ and $\frac1{n_i} \log\|s_{n_i}\| \longrightarrow 0$ in $L^1$-topology. Thanks to Lemma \ref{k^n then ok}, any sequence of sections in $B_{n_i}(s_{n_i},k^{n_i}) \cap H^0(\mathcal{X},n_i\mathcal{L})$ still has the desired convergence property at infinite places. The goal of this section is to find ``good'' sections in $B_{n_i}(s_{n_i},k^{n_i}) \cap H^0(\mathcal{X},n_i\mathcal{L})$ to meet other requirements. 


\subsection{Charles calculus}
We present several counting theorems of Charles \cite{charles_arithmetic_2021} with one minor generalization: Charles only considered balls that are centered at $0$ whereas in our application, balls are centered at arbitrary lattice point $s_n \in H^0(\mathcal{X},n\mathcal{L})$.

\begin{lem}[Charles \cite{charles_arithmetic_2021} Proposition 2.14]
	Let $\alpha$ and $\eta$ be positive real numbers with $0<\alpha<1$. Let $C$ be a real number. Then there exists a positive real number $\eta^\prime$ such that for any sequence of lattice points $s_n \in H^0(\mathcal{X},n\mathcal{L})$ and any $R_n>e^{-n^\alpha}$, we have
	\[
		\frac{ \left| \# \Big( B_n(s_n, R_n+Ce^{-\eta n}) \cap H^0(\mathcal{X},n\mathcal{L}) \Big) - \#\Big( B_n(s_n, R_n) \cap H^0(\mathcal{X},n\mathcal{L}) \Big) \right| }{ \# \Big( B_n(s_n, R_n) \cap H^0(\mathcal{X},n\mathcal{L}) \Big) } = O(e^{-\eta^\prime n})
	\]
	as $n \longrightarrow \infty$. The implicit constant in the big $O$ notation does not depend on $s_n$.
	\label{Charles bigger ball lemma}
\end{lem}

\begin{proof}
	Charles \cite{charles_arithmetic_2021}, Proposition 2.14 shows the lemma holds with center at $0$. When we consider balls centered at general lattice points $s_n$, there is no difference: the number of lattice points in a ball centered at any lattice point is equal. 
\end{proof}


\begin{thm}[Charles \cite{charles_arithmetic_2021} Theorem 2.21]
	If $\mathcal{Y} \subseteq \mathcal{X}$ is an irreducible closed subscheme, let
		\begin{flalign*}
		\quad \quad \quad &
		\operatorname{Res}_{n,\mathcal{Y}}: H^0(\mathcal{X},n\mathcal{L}) \longrightarrow H^0(\mathcal{Y},n\mathcal{L}).
		&&
	\end{flalign*}
	be the restriction map.
	
	There exists $\eta>0$ such that for any $s_n \in H^0(\mathcal{X},n\mathcal{L})$ and any irreducible closed subscheme $\mathcal{Y} \subsetneq \mathcal{X}$ of dimension $d>0$, we have
	\begin{flalign*}
		\quad \quad \quad &
		\frac{\# \Big(  B(s_n,1) \cap \ker(\operatorname{Res}_{n,\mathcal{Y}} ) \Big)}{\# \Big( B(s_n,1) \cap H^0(\mathcal{X},n\mathcal{L}) \Big) } = O(e^{-\eta n^d}).
		&&
	\end{flalign*} where the implied constant depend only on $\mathcal{X}$ and $\overline{\mathcal{L}}$ but not on $\mathcal{Y}$ nor $s_n$.
	\label{Charles Counting}
\end{thm}

\begin{remark}
	We write a detailed proof here for the readers' convenience. But this tiny generalization is really not substantial: we are just following Charles' proof step-by-step and changing the ball from $B(0,1)$ to $B(s_n,1)$. 
\end{remark}

\begin{proof}
	\underline{Horizontal case}. Assume $\mathcal{Y}$ is horizontal. Let $H_n$ be the kernel of the restriction map
	\begin{flalign*}
		\quad \quad \quad &
		\operatorname{Res}_{n,\mathcal{Y}}: H^0(\mathcal{X},n\mathcal{L}) \longrightarrow H^0(\mathcal{Y},n\mathcal{L})
		&&
	\end{flalign*} and let $k_n$ be the  corank of $H_n$. By Charles \cite{charles_arithmetic_2021}, Proposition 2.20 applied to $\mathcal{X}_\mathbb{Q}$ and $\mathcal{Y}_\mathbb{Q}$, there exists a positive constant $C$ independent of $\mathcal{Y}$, such that if $n$ is larger than some integer $N$, independent of $\mathcal{Y}$, then $k_n \geq Cn^{d-1}$.
	
	Up to enlarging $N$, by Zhang's Nakai-Moishezon when $n \geq N$, there exists a basis $H^0(\mathcal{X},n\mathcal{L})$ with norm at most $e^{-\varepsilon n}$. For $n \geq N$, we can find elements $\sigma_1,\dots,\sigma_{k_n}$ of $\Lambda_n$ that are linearly independent in $\Lambda_n/H_n$ and satisfy $\|\sigma_i\|_\infty \leq e^{-\varepsilon n}$ for $i=1,\dots,k_n$.
	
	Let $\eta$ be a positive number smaller than $\varepsilon$. Then for any $\sigma \in H_n \in B(s_n,1)$ and any integers $\lambda_1,\dots,\lambda_{k_n}$ with $|\lambda_i| \leq e^{\eta n}$, we have
	\begin{flalign*}
		\quad \quad \quad &
		\|\sigma +\sum_i \lambda_i \sigma_i - s_n \|_\infty \leq 1+k_n e^{-(\varepsilon-\eta)n}
		&&
	\end{flalign*}
	and, when $\sigma$ runs through $H_n \cap B(s_n,1)$ and $\lambda_i$ run through integers with $|\lambda_i| \leq e^{\eta n}$, the sections $\sigma +\sum_i \lambda_i \sigma_i$ are pairwisely distinct.
	
	As a result we have
	\begin{flalign*}
		\quad \quad \quad &
		e^{nk_n\eta} \# \Big( H_n \cap B(s_n,1) \Big) \leq \#\Big( H^0(\mathcal{X},n\mathcal{L}) \cap B(s_n,1+k_ne^{-(\varepsilon-\eta)n}) \Big).
		&&
	\end{flalign*} Note that $k_n$ is bounded by the rank of $H^0(\mathcal{X},n\mathcal{L})$, which is a polynomial, and apply \ref{Charles bigger ball lemma}, we get
	\begin{flalign*}
		\quad \quad \quad &
		\frac{\# \Big(  B(s_n,1) \cap \ker(\operatorname{Res}_{n,\mathcal{Y}} ) \Big)}{\# \Big( B(s_n,1) \cap H^0(\mathcal{X},n\mathcal{L}) \Big) } = O(e^{-nk_n\eta})
		&&
	\end{flalign*}
	$k_n \geq Cn^{d-1}$ so we get
	\begin{flalign*}
		\quad \quad \quad &
		\frac{\# \Big(  B(s_n,1) \cap \ker(\operatorname{Res}_{n,\mathcal{Y}} ) \Big)}{\# \Big( B(s_n,1) \cap H^0(\mathcal{X},n\mathcal{L}) \Big) } = O(e^{-\eta n^d})
		&&
	\end{flalign*} and we see from the above argument that the implied constants do not depend on $s_n$ nor $\mathcal{Y}$.

	\underline{Vertical case}.	Assume $\mathcal{Y}$ is vertical. By Charles \cite{charles_arithmetic_2021}, Proposition 2.20, there exists $C$ s.t. when $n$ is large (uniformly on $p$ and $\mathcal{Y} \subseteq \mathcal{X}_p$), the kernel $H_n$ of the restriction map
	\begin{flalign*}
		\quad \quad \quad &
		H^0(\mathcal{X},n\mathcal{L}) \longrightarrow H^0(\mathcal{Y},n\mathcal{L})
		&&
	\end{flalign*} has corank $k_n\geq Cn^d$.
	
	Let $\sigma \in H_n \cap B(s_n,1)$, and if $\lambda_1,\dots,\lambda_{k_n}$ be integers taken in $\big\{ 0,1,\dots,p \big\}$, then $\sigma + \sum \lambda_i \sigma_i$ are pairwisely distinct and $\sigma + \sum \lambda_i \sigma_i \in H_n$ $\Longleftrightarrow$ $\lambda_i$ are all zero. Consider only $\lambda_i=0$ or $1$, we have
	\begin{flalign*}
		\quad \quad \quad &
		2^{k_n} \cdot  \#\Big( B(s_n,1) \cap \ker (\operatorname{Res}_{n,\mathcal{Y}} ) \Big) \leq  \#\Big( B(s_n,1+k_n e^{-\varepsilon n}) \cap H^0(\mathcal{X},n\mathcal{L}) \Big).
		&&
	\end{flalign*}
	Apply \ref{Charles bigger ball lemma} and use $k_n \geq Cn^d$ we get 
	\begin{flalign*}
		\quad \quad \quad &
		\frac{\# \Big(  B(s_n,1) \cap \ker(\operatorname{Res}_{n,\mathcal{Y}} ) \Big)}{\# \Big( B(s_n,1) \cap H^0(\mathcal{X},n\mathcal{L}) \Big) } = O(e^{-\eta n^d}).
		&&
	\end{flalign*} and we see from the above argument that the implied constants do not depend on $s_n$ nor $\mathcal{Y}$.
\end{proof}

\begin{thm}[Charles \cite{charles_arithmetic_2021} Corollary 2.19]
	Let $\mathcal{Y}$ be a vertical closed subscheme and $s_n \in H^0(\mathcal{X},n\mathcal{L})$, then we have restriction map
	\begin{flalign*}
		\quad \quad \quad &
		\varphi_{s_n,\mathcal{Y}}: B(s_n,1) \cap H^0(\mathcal{X},n\mathcal{L}) \longrightarrow H^0(\mathcal{Y},n\mathcal{L}).
		&&
	\end{flalign*}
	Let $E_n \subseteq H^0(\mathcal{Y},n\mathcal{L})$ be subsets.
	If the proportion of $E_n $ in $ H^0(\mathcal{Y},n\mathcal{L})$ tends to $\rho$, then the proportion of $\varphi_{s_n,\mathcal{Y}}^{-1}(E_n) $ in $ B(s_n,1) \cap H^0(\mathcal{X},n\mathcal{L})$ also tends to $\rho$.
	\label{Charles Comparison}
\end{thm}

\begin{proof}
	This follows from Charles \cite{charles_arithmetic_2021} (where $s_n=0$) after a translation by $s_n$ in the left side and a translation by $\operatorname{Res}_{n,\mathcal{Y}}(s_n)$ in the right side.
\end{proof}

\subsection{Avoiding a closed subset}
\begin{prop}
	Let $Y$ be a closed subset of $\mathcal{X}_\mathbb{Q}$ and $s_n \in H^0(\mathcal{X},n\mathcal{L})$. Then the density of elements in $B(s_n,k^n) \cap H^0(\mathcal{X},n\mathcal{L})$ whose divisor do not contain any irreducible components of $Y$ tends to $1$.
\end{prop}

\begin{proof}
	Consider $\overline{\mathcal{L}}(\log k)=(\mathcal{L},\frac{1}{k}\|\cdot\|)$ which remains ample, we can replace  $\overline{\mathcal{L}}$ by $\overline{\mathcal{L}}(\log k)$ and we need only to show that the density of elements in $B(s_n,1) \cap H^0(\mathcal{X},n\mathcal{L})$ (with respect to the new metric) whose divisor do not contain any irreducible components of $Y$ tends to $1$. This follows immediately from Theorme \ref{Charles Counting}.
\end{proof}

\subsection{Generic smoothness}
Examining Charles' proof in \cite{charles_arithmetic_2021} \textsection 3.2 (where he considered $B(0,1)$ only), it actually works for $B(s_n,k^n)$.

We first state the Bertini smoothness theorem of Poonen \cite{poonen_bertini_2004} in the form we need: see \cite{erman_semiample_2015} for the proof of this version.

\begin{thm}
	Let $X$ be a smooth projective variety over finite field $k$ and $L$ be an ample line bundle on $X$. Then the density of elements in $H^0(X,nL)$ whose divisor are smooth tends to $\zeta_X(1+\dim X)^{-1}$.
	\label{Poonen Bertini Density}
\end{thm}

\begin{prop}
	Let $s_n \in H^0(\mathcal{X},n\mathcal{L})$. The density of elements in $B(s_n,k^n) \cap H^0(\mathcal{X},n\mathcal{L})$ whose divisors are generically smooth tends to $1$.
\end{prop}

\begin{proof}
By rescaling, we need only to show that the density of elements in $B(s_n,1) \cap H^0(\mathcal{X},n\mathcal{L})$ whose divisor are generically smooth tends to $1$.

	$\mathcal{X}/\mathbb{Z}$ is generically smooth so when $p$ is large, $\mathcal{X}_p/\mathbb{F}_p$ is smooth. By Theorem \ref{Charles Comparison} and Theorem \ref{Poonen Bertini Density}, the density of elements in $B(s_n,1) \cap H^0(\mathcal{X},n\mathcal{L})$ whose divisors are smooth at $p$ tends to $\zeta_{\mathcal{X}_p}(\dim \mathcal{X} )^{-1}$. Since the smooth locus is open, if $\div(\sigma)_p$ is smooth, then $\div(\sigma)$ is generically smooth. Thus the density of elements in $B(s_n,1) \cap H^0(\mathcal{X},n\mathcal{L})$ whose divisor are generically smooth is greater or equal to $\zeta_{\mathcal{X}_p}(\dim \mathcal{X} )^{-1}$ for any $p$ large. We win because (Serre \cite{serre_zeta_2003}, 1.3)
	\begin{flalign*}
		\quad \quad \quad &
		\lim_{p \longrightarrow \infty} \zeta_{\mathcal{X}_p}(\dim \mathcal{X} )^{-1} = 1.
		&&
	\end{flalign*}
\end{proof}

\subsection{No vertical components}
We follow the idea of Charles \cite{charles_arithmetic_2021} with a minor improvement in the relative dimension 1 case.

\begin{prop}
	Assume $\mathcal{X}/\mathbb{Z}$ is of relative dimension $d \geq 2$. Let $s_n \in H^0(\mathcal{X},n\mathcal{L})$ be such that $\|s_n\|_\infty \longrightarrow 1$. Then the density of elements in $B(s_n,k^n) \cap H^0(\mathcal{X},n\mathcal{L})$ whose divisor have no vertical components tends to $1$.
\end{prop}

\begin{proof}
	By rescaling, we need only to show that assume $\|s_n\|_\infty \leq A^n$ with $A>1$, then the density of elements in $B(s_n,1) \cap H^0(\mathcal{X},n\mathcal{L})$ whose divisor have no vertical components tends to $1$.
	
	\underline{Step I}. Take a horizontal curve $\mathcal{C}$. Then for large enough $p$, $\mathcal C$ meets every irreducible components of $\mathcal{X}_p$. Assume $\sigma \in B(s_n,1) \cap H^0(\mathcal{X},n\mathcal{L})$ is such that $\operatorname{div}(\sigma)$ does not contain $\mathcal{C}$ and has a vertical component over $p$. Then we may use $\sigma$ to compute the Arakelov degree of $\mathcal{C}$ and get
	\begin{flalign*}
		\quad \quad \quad &
		\widehat{\deg} \left( n\overline{\mathcal{L}}|_\mathcal{C} \right) = \operatorname{div}(\sigma) \cdot \mathcal{C} - \sum \log\|\sigma(x_\mathcal{C})\|.
		&&
	\end{flalign*}
	In particular,
	\begin{flalign*}
		\quad \quad \quad &
		\log p \leq n \widehat{\deg} \left(\overline{\mathcal{L}}|_\mathcal{C} \right) + n \deg(x_\mathcal{C}) \log (A+1) = nB
		&&
	\end{flalign*} where $B=\widehat{\deg} \left(\overline{\mathcal{L}}|_\mathcal{C} \right) + \deg(x_\mathcal{C}) \log (A+1)$ depends only on $\mathcal{C}$.
	In other words, if $\operatorname{div}( \sigma)$ has a vertical component over $p> e^{nB}$ then $\operatorname{div}( \sigma)$ must contain $\mathcal{C}$. Use Theorem \ref{Charles Counting} on $\mathcal{C}$ we see the density of elements in $B(s_n,1) \cap H^0(\mathcal{X},n\mathcal{L})$ whose divisor have vertical components above $ p> e^{nB}$ tends to $0$.
	
	\underline{Step II}. Fix a small $\varepsilon>0$. By Theorem \ref{Charles Counting}, the density of elements in $B(s_n,1) \cap H^0(\mathcal{X},n\mathcal{L})$ whose divisor have vertical components above $p \leq e^{\varepsilon n^2}$ is $O(e^{\varepsilon n^2 - \eta n^d})$. In particular it tends to $0$.

	\underline{Step III}. When $n$ is large, $e^{nB} < e^{\varepsilon n^2}$, which means that we have taken care of all $p$. We win.
\end{proof}

\begin{prop}
	Assume $\mathcal{X}/\mathbb{Z}$ is of relative dimension $d = 1$. Let $s_n \in H^0(\mathcal{X},n\mathcal{L})$ be such that $\|s_n\|_\infty \longrightarrow 1$. Then the density of elements in $B(s_n,k^n) \cap H^0(\mathcal{X},n\mathcal{L})$ whose divisor have no vertical components tends to $1$.
\end{prop}

\begin{proof}
By rescaling, we need only to show that assume $\|s_n\|_\infty \leq A^n$ with $A>1$, then the density of elements in $B(s_n,1) \cap H^0(\mathcal{X},n\mathcal{L})$ whose divisor have no vertical components tends to $1$.

\underline{Step I}. Take a horizontal curve $\mathcal{C}$. Then for large enough $p$, $\mathcal C$ meets every irreducible components of $\mathcal{X}_p$. Assume $\sigma \in B(s_n,1) \cap H^0(\mathcal{X},n\mathcal{L})$ is such that $\operatorname{div}(\sigma)$ does not contain $\mathcal{C}$ and has a vertical component over $p$. Then we may use $\sigma$ to compute the Arakelov degree of $\mathcal{C}$ and get
	\begin{flalign*}
		\quad \quad \quad &
		\widehat{\deg} \left( n\overline{\mathcal{L}}|_\mathcal{C} \right) = \operatorname{div}(\sigma) \cdot \mathcal{C} - \sum \log\|\sigma(x_\mathcal{C})\|.
		&&
	\end{flalign*}
	In particular,
	\begin{flalign*}
		\quad \quad \quad &
		\log p \leq n \widehat{\deg} \left(\overline{\mathcal{L}}|_\mathcal{C} \right) + n \deg(x_\mathcal{C}) \log (A+1) = nB
		&&
	\end{flalign*} where $B=\widehat{\deg} \left(\overline{\mathcal{L}}|_\mathcal{C} \right) + \deg(x_\mathcal{C}) \log (A+1)$ depends only on $\mathcal{C}$.
	In other words, if $\operatorname{div}( \sigma)$ has a vertical component over $p> e^{nB}$ then $\operatorname{div}( \sigma)$ must contain $\mathcal{C}$. Use Theorem \ref{Charles Counting} on $\mathcal{C}$ we see the density of elements in $B(s_n,1) \cap H^0(\mathcal{X},n\mathcal{L})$ whose divisor have vertical components above $ p> e^{nB}$ tends to $0$.
	
\underline{Step II}. Take a positive integer $q$ such that $q^{\deg(\mathcal{L}_\mathbb{Q})}>e^B$. Let $p$ be any prime number such that $q<p<e^{nB}$ and $\mathcal{Y}$ be any irreducible component of $\mathcal{X}_p$.

By Charles \cite{charles_arithmetic_2021}, Proposition 2.20, there exists $C$ such that when $n$ is large (uniformly on $p$ and $\mathcal{Y} \subseteq \mathcal{X}_p$), the kernel $H_n$ of the restriction map
\begin{flalign*}
	\quad \quad \quad &
	H^0(\mathcal{X},n\mathcal{L}) \longrightarrow H^0(\mathcal{Y},n\mathcal{L})
	&&
\end{flalign*} has corank $k_n\geq Cn$.

By Theorem \ref{Zhang Nakai-Moishezon}, when $n$ is large we can find a $\varepsilon>0$ such that there exist sections $\sigma_1,\dots,\sigma_{k_n} \in H^0(\mathcal{X},n\mathcal{L})$ such that $\|\sigma_i\|_\infty \leq e^{-\varepsilon n}$ and the images of $\sigma_1,\dots,\sigma_{k_n}$ in $H^0(\mathcal{Y},n\mathcal{L})$ are linearly independent over $\mathbb{F}_p$.

Let $\sigma \in H_n$, and if $\lambda_1,\dots,\lambda_{k_n}$ be integers taken in $\big\{ 0,1,\dots,q \big\}$, then $\sigma + \sum \lambda_i \sigma_i$ are pairwisely distinct and $\sigma + \sum \lambda_i \sigma_i \in H_n$ if and only if $\lambda_i$ are all zero.

Thus we have
\begin{flalign*}
	\quad \quad \quad &
	q^{r_n} \cdot  \#\Big( B(s_n,1) \cap \ker (\operatorname{Res}_{n,\mathcal{Y}} ) \Big) \leq  \#\Big( B(s_n,1+q Cn e^{-\varepsilon n}) \cap H^0(\mathcal{X},n\mathcal{L}) \Big).
	&&
\end{flalign*}
Thus for any $0<\varepsilon^\prime <\varepsilon$, we have
\begin{flalign*}
	\quad \quad \quad &
	\frac{ \#\Big( B(s_n,1) \cap \ker (\operatorname{Res}_{n,\mathcal{Y}} ) \Big) }{ \#\Big( B(s_n,1) \cap H^0(\mathcal{X},n\mathcal{L}) \Big) } \leq \frac1{q^{Cn}} \cdot \frac{ \#\Big( B(s_n,1+e^{- \varepsilon^\prime n}) \cap H^0(\mathcal{X},n\mathcal{L}) \Big) }{ \#\Big( B(s_n,1) \cap H^0(\mathcal{X},n\mathcal{L}) \Big) }.
	&&
\end{flalign*}
By Lemma \ref{Charles bigger ball lemma}, 
\begin{flalign*}
	\quad \quad \quad &
	\frac{ \#\Big( B(s_n,1+e^{- \varepsilon^\prime n}) \cap H^0(\mathcal{X},n\mathcal{L}) \Big) }{ \#\Big( B(s_n,1) \cap H^0(\mathcal{X},n\mathcal{L}) \Big) } \longrightarrow 1
	&&
\end{flalign*} hence

\begin{flalign*}
	\quad \quad \quad &
	\frac{ \#\Big( B(s_n,1) \cap \ker (\operatorname{Res}_{n,\mathcal{Y}} ) \Big) }{ \#\Big( B(s_n,1) \cap H^0(\mathcal{X},n\mathcal{L}) \Big) } \leq \frac2{q^{Cn}}.
	&&
\end{flalign*}
Take sum over all $q<p<e^{nB}$ and all $\mathcal{Y} \subseteq \mathcal{X}_p$, we see that the density of elements in $B(s_n,1) \cap H^0(\mathcal{X},n\mathcal{L})$ whose divisor have vertical components above $q<p<e^{nB}$ is bounded by $3 r \cdot \frac{e^{n B}}{q^{Cn}}$ where $r$ is the number of irreducible components of $\mathcal{X}_{\overline{\mathbb{Q}}}/\overline{\mathbb{Q}}$. We took $q$ such that $q^{\deg(\mathcal{L}_\mathbb{Q})}>e^B$, so this density tends to $0$. 

\underline{Step III}. Since $q$ depends only on $\mathcal{C}$ and not on $n$, the density of elements in $B(s_n,1) \cap H^0(\mathcal{X},n\mathcal{L})$ whose divisor have vertical components above $p \leq q$ tends to $0$.

\underline{Step IV}. We have taken care of $p \leq q$, $q < p < e^{n B}$ and $p \geq e^{n B}$. We win.
\end{proof}

\subsection{Summary}
The full arithmetic Demailly approximation theorem that we have been cooking is the following:
\begin{thm}
	Let $\mathcal{X}/\mathbb{Z}$ be an arithmetic variety, normal and generically smooth. Let $\overline{\mathcal{L}}$ be an ample Hermitian line bundle. Let $Y$ be a closed subset of $\mathcal{X}_\mathbb{Q}$.
	Then there exist $s_{n_i} \in H^0(\mathcal{X},n_i\mathcal{L})$ such that
	\begin{itemize}
		\item $\|s_{n_i}\|_\infty \longrightarrow 1$ and $\frac1{n_i} \log\|s_{n_i}\| \longrightarrow 0$ in $L^1$-topology.
		\item $\operatorname{div}(s_{n_i})_\mathbb{Q}$ is smooth and does not contain any irreducible components of $Y$.
		\item $\operatorname{div}(s_{n_i})$ has no vertical components.
	\end{itemize}
	\label{full arithmetic Demailly approximation}
\end{thm}

\begin{remark}
	Note that if relative dimension $\geq 2$, $\operatorname{div}(s_{n_i})_\mathbb{Q}$ is smooth and $\operatorname{div}(s_{n_i})$ has no vertical components will imply $\operatorname{div}(s_{n_i})$ is irreducible.
\end{remark}

\section{Application: $\operatorname{ess}(\overline N) \geq 0$ $\Longrightarrow $ $\overline N$ pseudo-effective}
\begin{thm}\label{main}
	Let $\mathcal{X}/\mathbb{Z}$ be an arithmetic variety of relative dimension $d$ and $\mathcal{\overline{N}}$ be a Hermitian line bundle. Then
	\begin{flalign*}
		\quad \quad \quad &
		\operatorname{ess}(\mathcal{\overline{N}}) \leq \inf \frac{\mathcal{ \overline{L}}^d \cdot \pi^*\mathcal{\overline{N}}}{\mathcal{L}_\mathbb{Q}^d},
		&&
	\end{flalign*}
	where the infimum is taken over all 
	\begin{itemize}
		\item $\pi: \mathcal{X}^\prime \longrightarrow \mathcal{X}$ is a birational morphism with $\mathcal{X}^\prime$ normal and generically smooth.
		\item $\mathcal{\overline{L}}$ is an ample Hermitian line bundle on $\mathcal{X}^\prime$.
	\end{itemize}
\end{thm}

\begin{proof}
	Fix a birational morphism $\pi\colon \mathcal{X}^\prime\longrightarrow \mathcal{X}$ with $\mathcal{X}^\prime$ normal and generically smooth, and an ample Hermitian line bundle $\overline{\overline{L}}$ on $\mathcal{X}^\prime$. 
	Fix a closed subset $Y \subseteq \mathcal{X}^\prime_\mathbb{Q}$. By Theorem \ref{full arithmetic Demailly approximation}, we can find a sequence of integral sections $s_{n_1} \in H^0(\mathcal{X}^\prime,n_1\mathcal{L})$ such that
	\begin{itemize}
		\item $\frac1{n_1} \log\|s_{n_1}\| \longrightarrow 0$ in $L^1$-topology (happens on $\mathcal{X}^\prime(\mathbb{C})$).
		\item $\operatorname{div}(s_{n_1})_\mathbb{Q}$ smooth and does not contain any irreducible components of $Y$.
		\item $\operatorname{div}(s_{n_1})$ is irreducible.
	\end{itemize}
	
	Then the intersection number is
	\begin{flalign*}
		\quad \quad \quad &
		\overline{\mathcal{L}}^d \cdot \pi^* \overline{\mathcal{N}} = \frac1{n_{1}} \cdot \overline{\mathcal{L}}^{d-1} \cdot \pi^*\overline{\mathcal{N}} \cdot \operatorname{div}(s_{n_{1}}) - \frac1{n_{1}} \int_{\mathcal{X}^\prime(\mathbb{C})} \log\|s_{n_{1}}\| c_1(\overline{\mathcal{L}})^{d-1}c_1(\pi^*\overline{\mathcal{N}}).
		&&
	\end{flalign*}
	
	For any $n_1$, we can find a sequence (indexed by $n_2$) of integral sections $s_{n_1,n_2} \in H^0(\operatorname{div}(s_{n_1}) ,n_2\mathcal{L})$ s.t.

	\begin{itemize}
		\item $\frac1{n_2} \log\|s_{n_1,n_2}\| \longrightarrow 0$ in $L^1$-topology (happens on $\operatorname{div}(s_{n_1}) (\mathbb{C})$).
		\item $\operatorname{div}(s_{n_1,n_2})_\mathbb{Q}$ smooth and does not contain any irreduicble components of $Y \cap \operatorname{div}(s_{n_1})$.
		\item $\operatorname{div}(s_{n_1,n_2})$ is irreducible.
	\end{itemize}
	
	Then the intersection number is
	\begin{flalign*}
		\quad \quad \quad 
		\overline{\mathcal{L}}^d \cdot \pi^*\overline{\mathcal{N}} &= \frac1{n_{1}n_{2}} \cdot \overline{\mathcal{L}}^{d-2} \cdot \pi^*\overline{\mathcal{N}} \cdot \operatorname{div}(s_{n_{1},n_{2}}) - \frac1{n_{1}} \int_{\mathcal{X}^\prime(\mathbb{C})} \log\|s_{n_{1}}\| c_1(\overline{\mathcal{L}})^{d-1}c_1(\pi^*\overline{\mathcal{N}}) \\
		&\quad \quad - \frac1{n_{1}n_{2}}\int_{\operatorname{div}(s_{n_{1}})(\mathbb{C})}\log\|s_{n_{1},n_{2}}\| c_1(\overline{\mathcal{L}})^{d-2}c_1(\pi^*\overline{\mathcal{N}}).
		&&
	\end{flalign*}	
		
	We repeat this process as long as relative dimension $\geq 2$, we get, for any fixed $n_1,\dots,n_{d-2}$, a sequence (indexed by $n_{d-1}$) of integral sections $s_{n_1,n_2,\dots,n_{d-1}} \in H^0(\operatorname{div}(s_{n_1,\dots,n_{d-2}}) ,n_{d-1}\mathcal{L})$ s.t.

	\begin{itemize}
		\item $\frac1{n_{d-1}} \log\|s_{n_1,\dots,n_{d-1}}\| \longrightarrow 0$ in $L^1$-topology (happens on $\operatorname{div}(s_{n_1,\dots,n_{d-2}}) (\mathbb{C})$).
		\item $\operatorname{div}(s_{n_1,\dots,n_{d-1}})_\mathbb{Q}$ smooth and does not contain any irreducible components of $Y \cap \operatorname{div}(s_{n_1,\dots,n_{d-2}})$.
		\item $\operatorname{div}(s_{n_1,\dots,n_{d-1}})$ is irreducible.
	\end{itemize}
	
	Now $\operatorname{div}(s_{n_1,\dots,n_{d-1}})/\mathbb{Z}$ has relative dimension $1$ and the intersections number is
	\begin{flalign*}
		\quad \quad \quad 
		\overline{\mathcal{L}}^d \cdot \pi^*\overline{\mathcal{N}} &= \frac1{n_{1}n_{2} \cdots n_{{d-1}}} \cdot \overline{\mathcal{L}} \cdot \pi^*\overline{\mathcal{N}} \cdot \operatorname{div}(s_{n_{1},\dots,n_{{d-1}}}) - \frac1{n_{1}} \int_{\mathcal{X}^\prime(\mathbb{C})} \log\|s_{n_{1}}\| c_1(\overline{\mathcal{L}})^{d-1}c_1(\pi^*\overline{\mathcal{N}}) \\
		&\quad \quad - \frac1{n_{1}n_{2}} \int_{\operatorname{div}(s_{n_{1}})(\mathbb{C})}\log\|s_{n_{1},n_{2}}\| c_1(\overline{\mathcal{L}})^{d-2}c_1(\pi^*\overline{\mathcal{N}}) \\
		&\quad \quad - \cdots\cdots \\
		&\quad \quad - \frac1{n_{1}n_{2} \cdots n_{{d-1}}} \int_{\operatorname{div}(s_{n_{1},\dots,n_{{d-2}}})(\mathbb{C})} \log\|s_{n_{1},\dots,n_{{d-1}}}\| c_1(\overline{\mathcal{L}}) c_1(\pi^*\overline{\mathcal{N}}).
		&&
	\end{flalign*}
	
	For any fixed $n_1,\dots,n_{d-1}$, we can find a sequence (indexed by $n_d$) of integral sections $s_{n_1,n_2,\dots,n_{d}} \in H^0(\operatorname{div}(s_{n_1,\dots,n_{d-1}}) ,n_{d}\mathcal{L})$ s.t. 
	
	\begin{itemize}
		\item $\frac1{n_{d}} \log\|s_{n_1,\dots,n_{d}}\| \longrightarrow 0$ in $L^1$-topology (happens on $\operatorname{div}(s_{n_1,\dots,n_{d-1}}) (\mathbb{C})$).
		\item $\operatorname{div}(s_{n_1,\dots,n_{d}})_\mathbb{Q}$ smooth and does not contain any irreducible components of $Y \cap \operatorname{div}(s_{n_1,\dots,n_{d-1}})$.
		\item $\operatorname{div}(s_{n_1,\dots,n_{d}})$ has no vertical components.
	\end{itemize}
	
	Now the intersection number is
	\begin{flalign*}
		\quad \quad \quad 
		\overline{\mathcal{L}}^d \cdot \pi^*\overline{\mathcal{N}} &= \frac1{n_{ 1}n_{ 2} \cdots n_{ d}} \cdot \pi^*\overline{\mathcal{N}} \cdot \operatorname{div}(s_{n_{ 1},\dots,n_{ d}}) - \frac1{n_{ 1}} \int_{\mathcal{X}^\prime(\mathbb{C})} \log\|s_{n_{ 1}}\| c_1(\overline{\mathcal{L}})^{d-1}c_1(\pi^*\overline{\mathcal{N}}) \\
		&\quad \quad - \frac1{n_{ 1}n_{ 2}} \int_{\operatorname{div}(s_{n_{ 1}})(\mathbb{C})}\log\|s_{n_{ 1},n_{ 2}}\| c_1(\overline{\mathcal{L}})^{d-2}c_1(\pi^*\overline{\mathcal{N}}) \\
		&\quad \quad - \cdots\cdots \\
		&\quad \quad - \frac1{n_{ 1}n_{ 2} \cdots n_{ d}} \int_{\operatorname{div}(s_{n_{ 1},\dots,n_{ {d-1}}})(\mathbb{C})} \log\|s_{n_{ 1},\dots,n_{ d}}\| c_1(\pi^*\overline{\mathcal{N}}).
		&&
	\end{flalign*}

	Say $\operatorname{div}(s_{n_1,\dots,n_d}) = \sum_i \overline{\big\{ x_i \big\}}$, where $x_i$ are closed points of the generic fiber.
	We have
	\begin{flalign*}
		\quad \quad \quad 
		\pi^*\overline{\mathcal{N}} \cdot \operatorname{div}(s_{n_{ 1},\dots,n_{ d}}) &= \sum_i \widehat{\deg}\left( \pi^*\overline{\mathcal{N}}|_{\overline{\{x_i\}}} \right) \\
		&= \sum_i \deg(x_i) h_{\pi^*\overline{\mathcal{N}}}(x_i).
		&&
	\end{flalign*}
	
	Let $x_{{n_1,\dots,n_d},0}$ be the point attaining the smallest $h_{\pi^*\overline{\mathcal{N}}}$-value among the decomposition $\operatorname{div}(s_{n_1,\dots,n_d}) = \sum_i \overline{\big\{ x_i \big\}}$, then
	\begin{flalign*}
		\quad \quad \quad 
		\pi^*\overline{\mathcal{N}} \cdot \operatorname{div}(s_{n_{ 1},\dots,n_{ d}}) &= \sum_i \deg(x_i) h_{\pi^*\overline{\mathcal{N}}}(x_i) \\
		&\geq \sum_i \deg(x_i) h_{\pi^*\overline{\mathcal{N}}}(x_{{n_1,\dots,n_d},0}) \\
		&= n_1\cdots n_d \mathcal{L}_\mathbb{Q}^d h_{\pi^*\overline{\mathcal{N}}}(x_{{n_1,\dots,n_d},0}).
		&&
	\end{flalign*}
	
	Thus, when $\min\{n_1,\dots,n_d\} \longrightarrow \infty$ we have 
	\begin{flalign*}
		\quad \quad \quad &
		\frac{\overline{\mathcal{L}^d} \cdot \pi^*\overline{\mathcal{N}}}{\mathcal{L}_\mathbb{Q}^d} \geq h_{\pi^*\overline{\mathcal{N}}}(x_{{n_1,\dots,n_d},0}) + o(1).
		&&
	\end{flalign*}
	
	We avoided $Y$, so 
	\begin{flalign*}
		\quad \quad \quad &
		\inf_{x \in (\mathcal{X}^\prime_\mathbb{Q}-Y)(\overline{\mathbb{Q}})} h_{\pi^*\overline{\mathcal{N}}}(x) \leq \liminf_{n_1,\dots,n_d \longrightarrow \infty} h_{\pi^*\overline{\mathcal{N}}}(x_{{n_1,\dots,n_d},0}) \leq \frac{\overline{\mathcal{L}^d} \cdot \pi^*\overline{\mathcal{N}}}{\mathcal{L}_\mathbb{Q}^d}.
		&&
	\end{flalign*}
	
	Take supremum among all closed subsets $Y \subsetneq \mathcal{X}^\prime_\mathbb{Q}$ we get
	\begin{flalign*}
		\quad \quad \quad &
		\operatorname{ess}(\overline{\mathcal{N}}) = \operatorname{ess}(\pi^*\overline{\mathcal{N}}) = \sup_Y \inf_{x \in (\mathcal{X}^\prime_\mathbb{Q}-Y)(\overline{\mathbb{Q}})} h_{\pi^*\overline{\mathcal{N}}}(x) \leq \frac{\overline{\mathcal{L}^d} \cdot \pi^*\overline{\mathcal{N}}}{\mathcal{L}_\mathbb{Q}^d}.
		&&
	\end{flalign*}
	
	Take infimum among all
	\begin{itemize}
		\item $\pi: \mathcal{X}^\prime \longrightarrow \mathcal{X}$ is a birational morphism with $\mathcal{X}^\prime$ normal and geometrically smooth.
		\item $\overline{\mathcal{L}}$ is an ample line bundle on $\mathcal{X}^\prime$.
	\end{itemize} We win.
\end{proof}

\begin{cor}
	Let $\mathcal{X}/\mathbb{Z}$ be an arithmetic variety and $\mathcal{\overline{N}}$ be a Hermitian line bundle. Then $\operatorname{ess}(\overline{\mathcal{N}}) \geq 0$ $\Longrightarrow$ $\overline{\mathcal{N}}$ is pseudo-effective.
	\label{pseff111}
\end{cor}

\begin{proof}
	We have $0 \leq \operatorname{ess}(\mathcal{\overline{N}}) \leq \inf \frac{\mathcal{ \overline{L}}^d \cdot \pi^*\mathcal{\overline{N}}}{\mathcal{L}_\mathbb{Q}^d}$. Then $\overline{\mathcal{N}}$ is pseudo-effective by the arithmetic BDPP theorem of Ikoma (\cite{ikoma_concavity_2015}, Theorem 6.4).
\end{proof}

For the passage from Hermitian line bundles to adelic ones, we employ \cite{ballay_successive_2021}, Lemma 4.3.
\begin{lem}
	Let $X/\mathbb{Q}$ be an integral projective variety and $\overline N=\big( N, \big\{\|\cdot\|_p\big\}_{p \leq \infty} \big)$ be an adelic line bundle that is not pseudo-effective. Then there exists a Hermitian model $(\mathcal{X},\overline{\mathcal{N}})$ of $(X,N)$ such that
	\begin{itemize}
		\item $\overline{\mathcal{N}}$ is not pseudo-effective.
		\item $\operatorname{ess}(\overline{\mathcal{N}}) > \operatorname{ess}(\overline{N})$.
	\end{itemize}
	\label{BallayReduction}
\end{lem}

\begin{proof}
	There exists a sequence $(\mathcal{X}_n,\overline{\mathcal{N}}_n)$ such that $\|\cdot\|_{n,p}/\|\cdot\|_p = 1$ for all but finitely many $p$ and $\|\cdot\|_{n,p}/\|\cdot\|_p \longrightarrow 1$ uniformly for the other $p$.
	
	Since $\overline{N}$ is not pseudo-effective (and that being not pseudo-effective is an open condition), there exists $\varepsilon>0$ such that when $n$ is large, $\overline{\mathcal{N}}_n(\varepsilon)$ is not pseudo-effective.
	
	For this $\varepsilon$, when $n$ is large we will have $\Big| h_{\overline{\mathcal{N}}_n} (x) - h_{\overline{N}}(x) \Big| < \varepsilon/2$ for every $x \in X(\overline{\mathbb{Q}})$, so $h_{\overline{\mathcal{N}}_n(\varepsilon)} (x) = h_{\overline{\mathcal{N}}_n} (x) + \varepsilon > h_{\overline{N}}(x) + \varepsilon/2$. In particular, $\operatorname{ess}(\overline{\mathcal{N}}_n(\varepsilon)) > \operatorname{ess}(\overline{N})$.
\end{proof}

\begin{cor}
\label{main2}
	Let $X/\mathbb{Q}$ be an integral projective variety and $\overline N$ be an adelic line bundle. Then $\operatorname{ess}(\overline N) \geq 0$ $\Longrightarrow $ $\overline N$ pseudo-effective.
\end{cor}

\begin{proof}
	If not, then by the above lemma we could produce a Hermitian line bundle which is not pseudo-effective but with positive essential minimum, contradicting Corollary \ref{pseff111}.
\end{proof}

\appendix
\section{Adelic line bundles on quasi-projective varieties}
The theory of adelic line bundles on quasi-projective varieties is developed in Yuan-Zhang's book \cite{yuan_adelic_2022} and played an important role in Yuan's proof of the uniform Bogomolov conjecture \cite{yuan_arithmetic_2022}.

We prove in this section that
\begin{thm}
	Let $X/\mathbb{Q}$ be an integral quasi-projective variety and $\overline{N}$ be an adelic line bundle. Then $\operatorname{ess}(\overline N) \geq 0$ $\Longrightarrow $ $\overline N$ pseudo-effective.
	\label{mainthm_quasi-projective}
\end{thm}

Readers should consult \cite{yuan_adelic_2022} for the definition of adelic line bundles (on quasi-projective varieties), volume and bigness.

\begin{definition}[pseudo-effective]
	Let $X/\mathbb{Q}$ be an integral quasi-projective variety and $\overline L$ be an adelic line bundle. We say $\overline L$ is \emph{pseudo-effective} if for any big adelic line bundle $\overline B$ and any $n>0$, $\overline{L}+\tfrac 1n \overline{B}$ is still big.
\end{definition}

\begin{lem}
	Say $\overline{L}$ is represented by $( \mathcal{X}_i, \overline{\mathcal{L}}_i )$ on a quasi-projective model $\mathcal{U}$. If $\overline{\mathcal{L}}_i$ is pseudo-effective (as a Hermitian line bundle, on the projective model $\mathcal{X}_i$) for all $i$, then $\overline{L}$ is pseudo-effective.
	\label{4}
\end{lem}

Theorem \ref{mainthm_quasi-projective} follows immediately after
\begin{lem}
	Let $X/\mathbb{Q}$ be an integral quasi-projective variety and $\overline N$ be an adelic line bundle that is not pseudo-effective. Then there exists a projective arithmetic variety $\mathcal{X}$ and a Hermitian line bundle $\overline{\mathcal{N}}$ on $\mathcal{X}$ such that
	\begin{itemize}
		\item $\overline{\mathcal{N}}$ is not pseudo-effective.
		\item $\operatorname{ess}(h_{ \overline{\mathcal{N}} }, \mathcal{X}_\mathbb{Q}) = \operatorname{ess}(h_{ \overline{\mathcal{N}} }, X) > \operatorname{ess}(h_{ \overline{N}}, X)$.
	\end{itemize}
\end{lem}

\begin{proof}
	$\overline{N}$ is not pseudo-effective so there exists $\varepsilon>0$ s.t. $\overline{N}(\varepsilon)$ is still not pseudo-effective. Say $\overline{N}$ is represented by $( \mathcal{X}_i, \overline{\mathcal{N}}_i )$ on a quasi-projective model $\mathcal{U}$. Then $\overline{N}(\varepsilon)$ is represented by $( \mathcal{X}_i, \overline{\mathcal{N}}_i(\varepsilon))$.
	
	When $i$ is large, $\overline{\mathcal{N}}_i(\varepsilon)$ is not pseudo-effective (as a Hermitian line bundle, on the projective model $\mathcal{X}_i$) by Lemma \ref{4}.
	
	For this $\varepsilon$, when $i$ is large, we will have $\Big| h_{\overline{N}}(x) - h_{\overline{\mathcal{N}}_i} \Big|< \varepsilon/2$ for any $x \in X(\overline{\mathbb{Q}})$, so in particular $\operatorname{ess}(h_{ \overline{\mathcal{N}}_i(\varepsilon)}, X) > \operatorname{ess}(h_{ \overline{N}}, X)$.
\end{proof}

\bibliography{ref}
\bibliographystyle{alphaurl}

\end{document}